\documentclass[letterpaper, 10 pt, conference]{ieeeconf}  

\IEEEoverridecommandlockouts                              
\usepackage[utf8]{inputenc}
\usepackage[T1]{fontenc}
\usepackage{amsmath}
\usepackage{mathrsfs}
\usepackage{amsfonts}
\usepackage{algorithm}
\usepackage{algorithmicx}
\usepackage{algpseudocode}

\usepackage{color}

\newtheorem{theorem}{\bf Theorem}[section]

\newtheorem{proposition}[theorem]{\bf Proposition}

\usepackage{graphicx}

\title{\LARGE \bf
Efficient Riccati recursion for
optimal control problems \\
with pure-state equality constraints 
}

\author{Sotaro Katayama$^{1}$ and Toshiyuki Ohtsuka$^{1}$
\thanks{$^{1}$S. Katayama and T. Ohtsuka are with Department of System Science, Graduate School of Informatics, Kyoto University, Kyoto, Japan
        {\tt\small katayama.25w@st.kyoto-u.ac.jp}, 
        {\tt\small ohtsuka@i.kyoto-u.ac.jp}}%
}

\begin{document}

\onecolumn
\noindent
© 2022 IEEE. Personal use of this material is permitted. Permission from IEEE must be obtained for all other uses, in any current or future media, including reprinting/republishing this material for advertising or promotional purposes, creating new collective works, for resale or redistribution to servers or lists, or reuse of any copyrighted component of this work in other works.

\hspace{1cm}

\noindent
\textbf{Published article:} \\ 
\noindent
S. Katayama and T. Ohtsuka, ``Efficient Riccati recursion for optimal control problems with pure-state equality constraints,'' 2022 American Control Conference (ACC), 2022, pp. 3579--3586, doi: 10.23919/ACC53348.2022.9867631.

\twocolumn

\maketitle
\thispagestyle{empty}
\pagestyle{empty}

\begin{abstract}

A novel approach to efficiently treat pure-state equality constraints in optimal control problems (OCPs) using a Riccati recursion algorithm is proposed.
The proposed method transforms a pure-state equality constraint into a mixed state-control constraint such that the constraint is expressed by variables at a certain previous time stage.
It is showed that if the solution satisfies the second-order sufficient conditions of the OCP with the transformed mixed state-control constraints, it is a local minimum of the OCP with the original pure-state constraints.
A Riccati recursion algorithm is derived to solve the OCP using the transformed constraints with linear time complexity in the grid number of the horizon, in contrast to a previous approach that scales cubically with respect to the total dimension of the pure-state equality constraints.
Numerical experiments on the whole-body optimal control of quadrupedal gaits that involve pure-state equality constraints owing to contact switches demonstrate the effectiveness of the proposed method over existing approaches.

\end{abstract}

\section{Introduction}
Optimal control underlies the motion planning and control of dynamical systems such as trajectory optimization (TO) and model predictive control (MPC) \cite{bib:mpcbook}.
TO achieves versatile and dynamically consistent planning by solving optimal control problems (OCPs). 
MPC leverages the same advantages as TO in real-time control by solving an OCP online within a particular sampling interval. 
It is essential, particularly for MPC, to solve direct OCPs within a short computational time, even if they involve highly complicated dynamics, a large dimensional state, and a long horizon.

Newton-type methods are the most practical methods used for solving OCPs in terms of the convergence speed. 
One of the most efficient algorithms that implement the Newton-type methods to solve both the single-shooting and multiple-shooting OCPs of large-scale systems is the Riccati recursion algorithm \cite{bib:mpcbook, bib:Giaf:2016}.
The Riccati recursion algorithm scales only linearly with respect to the number of discretization grids of the horizon, in contrast to the direct methods (i.e., methods applying the Cholesky decomposition directly to the entire Hessian matrix) that scale cubically.
For example, it was successfully applied to solve OCPs even for significantly complex systems such as legged robots with large degrees of freedom within very short computational times \cite{bib:legged, bib:legged2}.

However, there is a drawback of the Riccati recursion algorithm: it cannot efficiently treat pure-state equality constraints, which often arise, for example, in waypoint constraints, terminal constraints, switching constraints in hybrid systems such as legged robots \cite{bib:penalty, bib:DDP:jumpRobot}, and inequality constraints handled by active-set methods \cite{bib:constrainedRiccati}.
\cite{bib:constrainedRiccati} extended the Riccati recursion algorithm for pure-state equality constraints and illustrated its effectiveness over the direct method for certain quadratic programming problems.
However, the computational time of this method scales cubically with respect to the total dimension of pure-state equality constraints over the horizon, and it is inefficient when the total dimension is large.
\cite{bib:projection:LQR} proposed a projection approach to treat pure-state equality constraints with the Riccati recursion algorithm efficiently.
However, this approach can only treat the equality constraints whose relative degree is 1, for example, velocity-level constraints of second-order systems, and cannot treat position-level constraints for such systems, which is a very common and practical class of constraints.

Other popular constraint-handling methods used with the Riccati recursion algorithm are the penalty function method and the augmented Lagrangian (AL) method.
For example, \cite{bib:penalty} used the penalty function method and \cite{bib:DDP:jumpRobot} used the AL method to treat pure-state equality constraints representing the switching constraints arising in OCPs involving quadrupedal gaits.
A transformation of the linear-quadratic OCP into a dual problem for the efficient Riccati recursion \cite{bib:axehillPhDThesis} also used the penalty function method to treat the pure-state constraints.
However, the penalty function method practically yields only the approximated solution, as illustrated by the numerical results obtained in \cite{bib:penalty}.
The AL method can treat constraints better than the penalty function method, for example, it converges to the optimal solution even if the penalty parameter remains at a finite value. However, it generally lacks convergence speed compared with the Newton-type methods that achieve superlinear or quadratic convergence. 
The AL method essentially achieves linear convergence and it yields superlinear convergence only if the penalty parameter goes to infinity \cite{bib:Bertsekas:2016}, which is an impractical assumption.
For example, \cite{bib:DDP:jumpRobot} used the AL method to consider the switching constraints in an OCP of quadruped bouncing motion. The AL method required a large number of the iterations (up to 300), although a simple 2D robot model was used and only one cycle of the bouncing motion was considered, which led to very low-dimensional (only four dimensions in all) pure-state equality constraints.

In this paper, we propose a novel approach to efficiently treat pure-state equality constraints in OCPs using a Riccati recursion algorithm.
The proposed method transforms a pure-state equality constraint into a mixed state-control constraint such that the constraint is expressed by variables at a certain previous time stage.
We show a relationship between an OCP with the original pure-state constraint (the original OCP) and an OCP with the transformed mixed state-control constraint (the transformed OCP); if the solution satisfies the first-order necessary conditions (FONC) and/or second-order sufficient conditions (SOSC) of the transformed OCP, then the solution also satisfies the FONC and/or SOSC of the original OCP.
Therefore, if we find a solution that satisfies the SOSC of the transformed OCP, it is a local minimum of the original OCP.
We then derive a Riccati recursion algorithm to solve the transformed OCP with linear time complexity in the grid number of the horizon, in contrast with the previous approach \cite{bib:constrainedRiccati} that scales cubically with respect to the total dimension of pure-state equality constraints.
Moreover, because the proposed method is in substance a Newton's method for an optimization problem with equality constraints, the proposed method achieves superlinear or quadratic convergence, which distinguishes our approach from the penalty function method and the AL method in terms of convergence speed.
We present numerical experiments of the whole-body optimal control of quadrupedal gaits that involve pure-state equality constraints owing to contact switches, which represent the position-level constraints of a second-order system, and demonstrate the effectiveness of the proposed method over the existing approaches, that is, the approach of \cite{bib:constrainedRiccati} and the AL method.

This paper is organized as follows. 
In Section \ref{section:formulation}, we transform an OCP with a pure-state equality constraint into an OCP with a mixed state-control equality constraint. 
In Section \ref{section:Riccati}, we derive a Riccati recursion algorithm to apply Newton's method efficiently to the OCP with transformed mixed state-control equality constraints.  
In Section \ref{section:Proof}, the theoretical properties of the proposed transformation of OCPs are described.  
In Section \ref{section:numericalExperiments}, the proposed method is compared with existing methods and its effectiveness is demonstrated in terms of computational time and convergence speed. 
In Section \ref{section:conclu}, we conclude with a brief summary and mention of future work.

\textit{Notation:} We describe the partial derivatives of a differentiable function with respect to certain variables using a function with subscripts; that is, $f_x (x)$ denotes $\frac{\partial f}{\partial x} (x)$ and $g_{x y} (x, y)$ denotes $\frac{\partial^2 g}{\partial x \partial y} (x, y)$.

\section{Transformation of Optimal Control Problem with Pure-State Equality Constraints}\label{section:formulation}
\subsection{Original Optimal Control Problem}
We consider the following discrete-time OCP:
Find the state $x_0, ..., x_N \in \mathbb{R}^{n_x}$ and the control input $u_0, ..., u_{N-1} \in \mathbb{R}^{n_u}$ minimizing the cost function
\begin{equation}\label{eq:cost}
    J = \varphi (x_N) + \sum_{i=0}^{N-1} L (x_i, u_i) 
\end{equation}
subject to the state equation
\begin{equation}\label{eq:f}
    x_i + f (x_i, u_i) - x_{i+1} = 0 ,  \;\; i \in \left\{ 0, ..., N-1 \right\},
\end{equation}
a pure-state equality constraint
\begin{equation}\label{eq:phi}
    \phi (x_k) = 0, \; \phi (x_k) \in \mathbb{R}^{n_c},
\end{equation}
and the initial state constraint
\begin{equation}\label{eq:x0}
    x_0 - \bar{x} = 0, \; \bar{x} \in \mathbb{R}^{n_x}.
\end{equation}
In the following, we assume the form of the state as 
$x_i = \begin{bmatrix}
        q_i ^{\rm T} &
        v_i ^{\rm T}
\end{bmatrix} ^{\rm T}$,
where $q_i \in \mathbb{R}^n$ and $v_i \in \mathbb{R}^n$ represent the generalized coordinates and velocity of the system, respectively,
and assume the form of the state equation as follows:
\begin{equation}\label{eq:stateEquationForm}
    f (x_i, u_i) := 
    \begin{bmatrix}
        f ^{(q)} (x_i)  \\ 
        f ^{(v)} (x_i, u_i) 
    \end{bmatrix}, \;\; f ^{(q)} (x_i), \; f ^{(v)} (x_i, u_i) \in \mathbb{R}^n.
\end{equation}
We also assume that $k \geq 2$, $n_u \leq n$, $n_c \leq n$, and that the constraint (\ref{eq:phi}) depends only on the generalized coordinate, that is, its Jacobian is expressed as follows:
\begin{equation}\label{eq:positionLevelConstraints}
    \phi_x (x_k) = 
    \begin{bmatrix}
        \phi_q (q_k) & O
    \end{bmatrix}, \;\; \phi_q (q_k) \in \mathbb{R}^{n_c \times n}.
\end{equation}
The state equation (\ref{eq:stateEquationForm}) mainly represents a second-order Lagrangian system with $n$ degrees of freedom, and a constraint (\ref{eq:phi}), whose Jacobian is of form (\ref{eq:positionLevelConstraints}) represents a position-level constraint (that is, the relative degree of the constraint with respect to the control input is 2), which is a very common and practical class of problem settings.

\subsection{Transformation of Optimal Control Problem}
To solve the aforementioned OCP efficiently, we transform the original pure-state equality constraint (\ref{eq:phi}) into a mixed state-control equality constraint that is equivalent to (\ref{eq:phi}) as long as (\ref{eq:f}) is satisfied.
If (\ref{eq:f}) is satisfied, 
\begin{align*}
    x_{k} = \; & x_{k-2} + f (x_{k-2}, u_{k-2} ) \\
               & + f (x_{k-2} + f (x_{k-2}, u_{k-2}), u_{k-1} ) 
\end{align*}
holds and therefore
\begin{align}\label{eq:phi:transformed0}
    \phi ( & x_{k-2} + f (x_{k-2}, u_{k-2} ) \notag \\
           & + f (x_{k-2} + f (x_{k-2}, u_{k-2}), u_{k-1} )) = 0 
\end{align}
is equivalent to (\ref{eq:phi}).
Furthermore, because $\phi (\cdot)$ only depends on the generalized coordinate, (\ref{eq:phi:transformed0}) is equivalent to
\begin{equation}\label{eq:phi:relaxed}
    \phi (x_{k-2} + f (x_{k-2}, u_{k-2} ) + g (x_{k-2} + f (x_{k-2}, u_{k-2}) ) ) = 0,
\end{equation}
where we define 
\begin{equation}\label{eq:g}
    g(x) := \begin{bmatrix}
        f^{(q)} (x) \\
        0
    \end{bmatrix}.
\end{equation}
Therefore, the constraint (\ref{eq:phi:relaxed}) is equivalent to (\ref{eq:phi}) if (\ref{eq:f}) is satisfied.
Therefore, we consider (\ref{eq:phi:relaxed}) instead of (\ref{eq:phi}) in the following. 
We herein summarize the original and transformed OCPs as follows:

{\bf Problem 1 – Original OCP:}
Find the solution $x_0, ..., x_N$, $u_0, ..., u_{N-1}$ minimizing (\ref{eq:cost}) subject to (\ref{eq:f})--(\ref{eq:x0}).

{\bf Problem 2 - Transformed OCP:}
Find the solution $x_0, ..., x_N$, $u_0, ..., u_{N-1}$ minimizing (\ref{eq:cost}) subject to (\ref{eq:f}), (\ref{eq:x0}), and (\ref{eq:phi:relaxed}).

In fact, we have the following relations between the transformed OCP and the original OCP:
If the solution $x_0, ..., x_N$, $u_0, ..., u_{N-1}$ satisfies the FONC of the transformed OCP, it also satisfies the FONC of the original OCP. If the solution $x_0, ..., x_N$, $u_0, ..., u_{N-1}$ satisfies the SOSC of the transformed OCP, it also satisfies the SOSC of the original OCP.
Therefore, if we find the solution $x_0, ..., x_N$, $u_0, ..., u_{N-1}$ that satisfies the SOSC of the transformed OCP, it is a local minimum of the original OCP.
We show these theoretical points later in Section \ref{section:Proof}.

It should be noted that it is trivial to apply the proposed approach to constraints of relative degree 1: we transform (\ref{eq:phi}) into a mixed state-control constraint represented by $x_{k-1}$ and $u_{k-1}$ in the same manner as mentioned before. 
Therefore, our approach comprises the approach of \cite{bib:projection:LQR} that also involves constraint transformation but can treat only constraints of relative degree 1. 
The difference between our approach and that of \cite{bib:projection:LQR} is that we introduce the constraint transformation in the original nonlinear problem and leverage the structure of the system (\ref{eq:f}), whereas in approach \cite{bib:projection:LQR}, constraint transformation is introduced in the linear subproblem arising in the Newton-type iterations without any assumptions in the state equation.
As a result, only our approach can treat the practically important constraints of relative degree 2 with a theoretical justification for the transformation.

It should also be noted that the proposed approach is completely different from the classical transformation of pure-state equality constraints in continuous-time OCPs by considering their derivatives with respect to time, for example, in Section 3.4 of \cite{bib:Bryson:1975}.
To explain this difference, we consider that there is a pure-state constraint of the form of (\ref{eq:phi}) over a time interval.
The classical method then transforms the pure-state constraint over the interval into a combination of the pure-state equality constraints (\ref{eq:phi}) and $\frac{d}{dt}\phi (x) = 0$ at a point on the interval and the mixed state-control constraint $\frac{d^2}{dt^2} \phi (x) = 0$ over the interval, where $\frac{d}{dt}$ and $\frac{d^2}{dt^2}$ yield a kind of Lie derivative. 
Therefore, the classical method still needs to consider the pure-state equality constraints, whereas our approach transforms all pure-state equality constraints into the corresponding mixed state-control constraints.

\subsection{Optimality Conditions}
We derive the optimality conditions, known as FONC, of the transformed OCP.
We first define the Hamiltonian
\begin{equation*}
    H (x, u, \lambda) := L (x, u) + \lambda ^{\rm T} f (x, u)
\end{equation*}
and 
\begin{align*}
    & \tilde{H} (x, u, \lambda, \nu) \notag \\ 
    & := H (x, u, \lambda) + \nu ^{\rm T} \phi (x + f(x, u) + g(x + f (x, u))).
\end{align*}
We also define the intermediate time stages in which the constraint is not active as $\bar{I} := \left\{1, ..., k-3, k-1, ..., N -1 \right\}$.
The optimality conditions are then derived as follows \cite{bib:Bryson:1975}:
\begin{equation}\label{eq:KKT:phix}
    \varphi_{x} ^{\rm T} (x_N) - \lambda_N = 0,
\end{equation}
\begin{equation}\label{eq:KKT:Hx}
    H_x ^{\rm T} (x_i, u_i, \lambda_{i+1}) + \lambda_{i+1} - \lambda_i = 0
\end{equation}
and
\begin{equation}\label{eq:KKT:Hu}
    H_u ^{\rm T} (x_i, u_i, \lambda_{i+1}) = 0
\end{equation}
for $i \in \bar{I}$,
\begin{align}\label{eq:KKT:Hx:k-2}
    & \tilde{H}_x ^{\rm T} (x_{k-2}, u_{k-2}, \lambda_{k-1}, \nu) + \lambda_{k-1} - \lambda_{k-2} \notag \\
    & = {H}_x ^{\rm T} (x_{k-2}, u_{k-2}, \lambda_{k-1}) + \lambda_{k-1} - \lambda_{k-2} \notag \\ 
    & \;\;\;\;\; + (I + f_x ^{\rm T} (x_{k-2}, u_{k-2})) (I + g_x ^{\rm T}) \phi_x ^{\rm T} \nu = 0,  
\end{align}
and
\begin{align}\label{eq:KKT:Hu:k-2}
    & \tilde{H}_u ^{\rm T} (x_{k-2}, u_{k-2}, \lambda_{k-1}, \nu) \notag \\ 
    & = {H}_u ^{\rm T} (x_{k-2}, u_{k-2}, \lambda_{k-1}) \notag \\ 
    & \;\;\;\;\; + f_u ^{\rm T} (x_{k-2}, u_{k-2}) (I + g_x ^{\rm T}) \phi_x ^{\rm T} \nu = 0, 
\end{align}
where $\lambda_0, ..., \lambda_N$ are the Lagrange multipliers with respect to (\ref{eq:x0}) and (\ref{eq:f}), and $\nu$ is that with respect to (\ref{eq:phi:relaxed}).
It should be noted that we have omitted the arguments from $\phi_x $ and $g_x$ in (\ref{eq:KKT:Hx:k-2}) and (\ref{eq:KKT:Hu:k-2}).

\section{Riccati Recursion}\label{section:Riccati}
\subsection{Linearization for Newton's Method}
To apply Newton's method for the transformed OCP, we linearize the optimality conditions.
It should be noted that we adopt the direct multiple shooting method \cite{bib:condensing}, that is, we consider $x_0, ..., x_N$, $u_0, ..., u_{N-1}$, $\lambda_0, ..., \lambda_N$, and $\nu$ as the optimization variables.

\subsubsection{Terminal stage}
At the terminal stage ($i=N$), we have 
\begin{equation}\label{eq:newton:QxxN}
    Q_{xx, N} \Delta x_N - \Delta \lambda_{N} + \bar{l}_{x, N} = 0,
\end{equation}
where we define $Q_{xx, N} := \varphi_{xx} (x_N)$. Further, we define $\bar{l}_{x, N}$ using the left-hand side of (\ref{eq:KKT:phix}).

\subsubsection{Intermediate stages without equality constraint}
In the intermediate stages without an equality constraint ($i \in \bar{I}$), we have
\begin{equation}\label{eq:newton:xi}
    A_i \Delta x_i + B_i \Delta u_i - \Delta x_{i+1} + \bar{x}_i = 0, 
\end{equation}
\begin{equation}\label{eq:newton:Qxxi}
    Q_{xx, i} \Delta x_i + Q_{xu, i} \Delta u_i + A_i ^{\rm T} \Delta \lambda_{i+1}
    - \Delta \lambda_i + \bar{l}_{x, i} = 0,  
\end{equation}
and
\begin{equation}\label{eq:newton:Quui}
    Q_{xu, i} ^{\rm T} \Delta x_i + Q_{uu, i} \Delta u_i + B_i ^{\rm T} \Delta \lambda_{i+1}
    + \bar{l}_{u, i} = 0,
\end{equation}
where we define 
$A_i := I + f_x (x_i, u_i) $, 
$B_i \allowbreak := f_u (x_i, u_i) $, 
$Q_{xx, i} \allowbreak := H_{xx} (x_i, u_i, \lambda_{i+1})$,
$Q_{xu, i} \allowbreak := H_{xu} (x_i, u_i, \lambda_{i+1})$,
$Q_{uu, i} \allowbreak := H_{uu} (x_i, u_i, \lambda_{i+1})$.
Further, we define $\bar{x}_{i}$, $\bar{l}_{x, i}$, and $\bar{l}_{u, i}$ using the left-hand sides of (\ref{eq:f}), (\ref{eq:KKT:Hx}), and (\ref{eq:KKT:Hu}), respectively. 

\subsubsection{Intermediate stage with an equality constraint}
At the intermediate stages with an equality constraint ($i = k -2$), we have
(\ref{eq:newton:xi}) for $i = k-2$ and have
\begin{align}\label{eq:newton:Qxxk-2}
    Q_{xx, k-2} \Delta x_{k-2} & + Q_{xu, k-2} \Delta u_{k-2} + A_{k-2} ^{\rm T} \Delta \lambda_{k-1} \notag \\ 
    & - \Delta \lambda_{k-2} + C ^{\rm T} \Delta \nu + \bar{l}_{x, k-2} = 0, 
\end{align}
\begin{align}\label{eq:newton:Quuk-2}
    Q_{xu, k-2} ^{\rm T} \Delta x_{k-2} + Q_{uu, k-2} \Delta u_{k-2} + B_{k-2} ^{\rm T} \Delta \lambda_{k-1} \notag \\ 
    + D ^{\rm T} \Delta \nu + \bar{l}_{u, k-2} = 0, 
\end{align}
and
\begin{equation}\label{eq:newton:psi}
    C \Delta x_{k-2} + D \Delta u_{k-2} + \bar{\phi} = 0,
\end{equation}
where we define 
$Q_{xx, k-2} \allowbreak := \tilde{H}_{xx} (x_{k-2}, u_{k-2}, \lambda_{k-1}, \nu) $,
$Q_{xu, k-2} \allowbreak := \tilde{H}_{xu} (x_{k-2}, u_{k-2}, \lambda_{k-1}, \nu) $,
$Q_{uu, k-2} \allowbreak := \tilde{H}_{uu} (x_{k-2}, u_{k-2}, \lambda_{k-1}, \nu) $,
$C \allowbreak := \phi_x (I + g_x) A_{k-2}$,
$D \allowbreak := \phi_x (I + g_x) B_{k-2}$.
We further define $\bar{l}_{x, k-2}$, $\bar{l}_{u, k-2}$, and $\bar{\phi}$ using the left-hand sides of (\ref{eq:KKT:Hx:k-2}), (\ref{eq:KKT:Hu:k-2}), and (\ref{eq:phi:relaxed}), respectively.

\subsubsection{Initial stage}
Finally, we have 
\begin{equation}\label{eq:newton:x0}
    \Delta x_0 + x_0 - \bar{x} = 0.
\end{equation}

It should be noted that we can apply the Gauss-Newton Hessian approximation, which improves the computational speed when the constraints (\ref{eq:f}) and (\ref{eq:phi:relaxed}) are too complicated for their second-order derivatives to be computed.
$Q_{xx, N}$ and $Q_{xx, i}$, $Q_{xu, i}$, and $Q_{uu, i}$ for $i \in \left\{ 0, ..., N-1 \right\}$ are then approximated using only the cost function (\ref{eq:cost}) and do not depend on the Lagrange multipliers.

\subsection{Derivation of Riccati Recursion}
We derive a Riccati recursion algorithm to solve the linear equations for Newton's method (\ref{eq:newton:QxxN})--(\ref{eq:newton:x0}) efficiently.
As the standard Riccati recursion algorithm (\cite{bib:mpcbook, bib:Giaf:2016}), our goal is to derive a series of matrices $P_i$ and vectors $s_i$ such that 
\begin{equation}\label{eq:riccati:lmdi}
    \Delta \lambda_i = P_i \Delta x_i - s_i  
\end{equation}
holds.

\subsubsection{Terminal stage} 
At the terminal stage ($i = N$), we have
\begin{equation}\label{eq:riccati:N}
    P_N = Q_{xx, N}, \; s_N = - \bar{l}_N.
\end{equation}
In the forward recursion, we have $\Delta x_N$, and we compute $\Delta \lambda_N$ from (\ref{eq:riccati:lmdi}).

\subsubsection{Intermediate stages without an equality constraint}
At the intermediate stages without an equality constraint ($i \in \bar{I}$), we have the following standard backward Riccati recursion (\cite{bib:mpcbook, bib:Giaf:2016}) for given $P_{i+1}$ and $s_{i+1}$ satisfying (\ref{eq:riccati:lmdi}):
\begin{equation}\label{eq:riccati:Fi}
    F_i := Q_{xx, i} + A_i ^{\rm T} P_{i+1} A_i,
\end{equation}
\begin{equation}\label{eq:riccati:Hi}
    H_i := Q_{xu, i} + A_i ^{\rm T} P_{i+1} B_i,
\end{equation}
\begin{equation}\label{eq:riccati:Gi}
    G_i := Q_{uu, i} + B_i ^{\rm T} P_{i+1} B_i,
\end{equation}
\begin{equation}\label{eq:riccati:Ki}
    K_i := - G_i ^{-1} H_i ^{\rm T}, \; k_i := - G_i ^{-1} (B_i ^{\rm T} P_{i+1} \bar{x}_i - B_i ^{\rm T} s_{i+1} + \bar{l}_{u, i}),
\end{equation}
and
\begin{equation}\label{eq:riccati:Pi}
    P_i := F_i - K_i ^{\rm T} G_{i} K_i, \; s_i := A_i ^{\rm T} (s_{i+1} - P_{i+1} \bar{x}_i) - \bar{l}_{x, i} - H_i k_i.
\end{equation}
In the forward recursion, for a particular value of $\Delta x_i$, we compute $\Delta u_i$ from  
\begin{equation}\label{eq:riccati:ui}
    \Delta u_i = K_i \Delta x_i + k_i,  
\end{equation}
$\Delta \lambda_i$ from (\ref{eq:riccati:lmdi}), and $\Delta x_{i+1}$ from (\ref{eq:newton:xi}).

\subsubsection{Intermediate stage with an equality constraint}
At the intermediate stage with an equality constraint ($i = k-2 $), we first define (\ref{eq:riccati:Fi})--(\ref{eq:riccati:Pi}) for $i=k-2$ for the specific values of $P_{k-1}$ and $s_{k-1}$ that satisfies (\ref{eq:riccati:lmdi}).
We then have the relations that are used in the forward recursion for $k-1$ as follows:
\begin{equation}\label{eq:riccati:uk-2nu}
    \begin{bmatrix}
        \Delta u_{k-2} \\
        \Delta \nu
    \end{bmatrix}
    = 
    \begin{bmatrix}
        K_{k-2} \\
        M 
    \end{bmatrix}
    \Delta x_{k-2}
    + 
    \begin{bmatrix}
        k_{k-2} \\
        m 
    \end{bmatrix},
\end{equation}
where we define
\begin{equation}\label{eq:riccati:KM}
    \begin{bmatrix}
        K_{k-2} \\
        M 
    \end{bmatrix}
    := - \begin{bmatrix}
        G_{k-2} &  D ^{\rm T} \\
        D & O
    \end{bmatrix}^{-1}
    \begin{bmatrix}
        H_{k-2} ^{\rm T} \\
        C 
    \end{bmatrix}
\end{equation}
and
\begin{align}\label{eq:riccati:km}
    \begin{bmatrix}
        k_{k-2} \\
        m
    \end{bmatrix}
    := - & \begin{bmatrix}
        G_{k-2} & D ^{\rm T} \\
        D & O
    \end{bmatrix}^{-1} \notag \\
    & \begin{bmatrix}
        B_{k-2} ^{\rm T} P_{k-1} \bar{x}_{k-2} - B_{k-2} ^{\rm T} s_{k-1} + \bar{l}_{u, k-2} \\
        \bar{\phi}
    \end{bmatrix}.
\end{align}
We then obtain the backward recursions
\begin{equation}\label{eq:riccati:Pk-2}
    P_{k-2} := F_{k-2} - \begin{bmatrix} K_{k-2} ^{\rm T} & M ^{\rm T} \end{bmatrix} 
    \begin{bmatrix}
        G_{k-2} & D ^{\rm T} \\
        D & O
    \end{bmatrix}
    \begin{bmatrix}
        K_{k-2} \\
        M 
    \end{bmatrix} 
\end{equation}
and 
\begin{align}\label{eq:riccati:sk-2}
    s_{k-2} := \; & A_{k-2} ^{\rm T} (s_{k-1} - P_{k-1} \bar{x}_{k-2}) - \bar{l}_{x, k-2} - H_{k-2} k_{k-2} \notag \\ 
    & - C ^{\rm T} m.
\end{align}

\subsection{Algorithm, Convergence, and Computational Analysis}
We summarize the single Newton iteration, that is, the computation of the Newton direction for a particular solution, using the proposed Riccati recursion algorithm  (Algorithm \ref{alg1}).
In the first step, we formulate the linear equations of Newton's method, that is, we compute the coefficient matrices and residuals of (\ref{eq:newton:QxxN})--(\ref{eq:newton:x0}) (line 1).
This step can leverage parallel computing.
Second, we perform the backward Riccati recursion and compute $P_i$ and $z_i$ for $i \in \left\{ 0, ..., N \right\}$ (lines 4--11).
Third, we perform the forward Riccati recursion and compute the Newton directions for all the variables (lines 12--20).

Because the proposed method is in substance a Newton's method for an optimization problem with equality constraints, it achieves superlinear or quadratic convergence, for example, by Proposition 4.4.3 of \cite{bib:Bertsekas:2016},
which distinguishes the convergence behavior of the proposed method from that of the AL method, popularly used to treat the pure-state equality constraints with the Riccati recursion algorithm.
The AL method achieves superlinear convergence only if the penalty parameter goes to infinity, which is an impractical assumption; otherwise, its convergence rate is just linear.

It should be noted that we can trivially apply the proposed method to OCPs with multiple pure-state equality constraints on the horizon.
When there are multiple time stages involving constraint (\ref{eq:phi}) on the horizon, we compute the coefficient matrices and residuals in (\ref{eq:newton:Qxxk-2})--(\ref{eq:newton:psi}) for each of the time stages in line 1 of Algorithm \ref{alg1}, apply line 7 of Algorithm \ref{alg1} for each of the time stages in the backward Riccati recursion, and apply line 15 of Algorithm \ref{alg1} for each of the time stages in the forward Riccati recursion.

The proposed method is particularly efficient when there are several stages with pure-state equality constraints on the horizon.
Suppose that there is an $n_{c, i}$-dimensional pure-state equality constraint at each time stage $i$ of the horizon ($n_{c, i} = 0$ if there is no constraint at stage $i$).
The proposed method then computes the inverse of a matrix whose size is $(n_u + n_{c, i}) \times (n_u + n_{c, i})$ at each time stage in the backward recursion.
In contrast, the previous approach of \cite{bib:constrainedRiccati} requires the computation of the inverse of a matrix of size $(\sum_{i=0}^{N} n_{c, i}) \times (\sum_{i=0}^{N} n_{c, i})$.
Broadly speaking, the computational burden of the proposed method with respect to the grid number $N$ is $O (N)$, whereas that of the approach in \cite{bib:constrainedRiccati} is $O(N ^{3})$.

It should be noted that it is easy to apply the proposed method to the single-shooting methods, which are popular in robotic applications \cite{bib:legged, bib:legged2,  bib:penalty, bib:DDP:jumpRobot}, by considering only $u_0, ..., u_{N-1}$ and $\nu$ as the decision variables.
In the single-shooting methods, before line 1 of Algorithm 1, we first compute $x_0, ..., x_N$ based on $x(t_0)$ and $u_0, ..., u_{N-1}$ using the state equation (\ref{eq:f}) sequentially.
Further, we compute $\lambda_N, ..., \lambda_0$ using (\ref{eq:KKT:phix}), (\ref{eq:KKT:Hx}), and (\ref{eq:KKT:Hx:k-2}) based on $u_0, ..., u_{N-1}$, $\nu$, and $x_0, ..., x_N$, respectively, in the backward recursion (lines 5--11 of Algorithm 1).
We can then compute the Newton directions $\Delta u_0, ..., \Delta u_{N-1}$ and $\Delta \nu$ using the same (or a similar) forward recursion (lines 12--20 of Algorithm 1).

\begin{algorithm}[tb]
\caption{Computation of Newton direction using proposed Riccati recursion}
\label{alg1}
\begin{algorithmic}[1]
    \Require Initial state ${x} (t_0)$, the current solution $x_0, ..., x_{N}$, $u_0, ..., u_{N-1}$, and Lagrange multipliers $\lambda_0, ..., \lambda_N$, $\nu$.
    \Ensure Newton directions $\Delta x_0$, ..., $\Delta x_{N}$, $\Delta u_0$, ..., $\Delta u_{N-1}$, $\Delta \lambda_0$, ..., $\Delta \lambda_N$, and $\Delta \nu$.
    \For{$i=0,\cdots,N$} {\bf in parallel} 
    \State Computes the matrices and residuals in (\ref{eq:newton:QxxN})--(\ref{eq:newton:x0}).
    \EndFor
    \State Compute $P_N$ and $z_N$ from (\ref{eq:riccati:N}). 
    \For{$i=N,\cdots,0$} {\bf in serial}
        \If{$i = k-2$} 
            \State Computes $P_{k-2}$ and $z_{k-2}$ from (\ref{eq:riccati:Fi})--(\ref{eq:riccati:Gi}), (\ref{eq:riccati:KM}), (\ref{eq:riccati:km}), and (\ref{eq:riccati:sk-2}). 
        \Else
            \State Computes $P_i$ and $z_i$ from (\ref{eq:riccati:Fi})--(\ref{eq:riccati:Pi}).
        \EndIf
    \EndFor
    \State Compute $\Delta x_0$ from (\ref{eq:newton:x0}).
    \For{$i=0,\cdots,N-1$} {\bf in serial}
        \If{$i = k-2$}
            \State Compute $\Delta u_{k-2}$, $\Delta \nu$, $\Delta \lambda_{k-2}$, and $\Delta x_{k-1}$ from (\ref{eq:riccati:uk-2nu}), (\ref{eq:riccati:lmdi}), and (\ref{eq:newton:xi}).
        \Else
            \State Compute $\Delta u_i$, $\Delta \lambda_i$, and $\Delta x_{i+1}$ from (\ref{eq:riccati:ui}), (\ref{eq:riccati:lmdi}), and (\ref{eq:newton:xi}).
        \EndIf
    \EndFor
    \State Compute $\Delta \lambda_N$ from (\ref{eq:riccati:lmdi}).
\end{algorithmic}
\end{algorithm}

\section{Theoretical Properties of Optimal Control Problem Transformation}\label{section:Proof}

We show the theoretical relationships between the transformed OCP and the original OCP in this section.
The first theorem concerns a stationary point of the transformed OCP and a stationary point of the original OCP.

\begin{theorem}
Suppose that $x_0, ..., x_N$, $u_0, ..., u_{N-1}$, $\lambda_0, ..., \lambda_N$, and $\nu$ satisfy the FONC of the transformed OCP.
Then, there exist the Lagrange multipliers $\lambda_0 ^*, ..., \lambda_N ^*$ and $\nu ^*$ such that $x_0, ..., x_N$, $u_0, ..., u_{N-1}$, $\lambda_0 ^*, ..., \lambda_N ^*$, and $\nu ^*$ satisfy the FONC of the original OCP. 
\end{theorem}

\begin{proof}
First, we define the intermediate time stages without the active constraints of the original OCP as $\tilde{I} := \left\{ 0,...,k-1, k+1, ..., N-1 \right\}$.
The FONC of the original OCP is then expressed by (\ref{eq:f})--(\ref{eq:x0}), (\ref{eq:KKT:phix}), and (\ref{eq:KKT:Hx}) for $i \in \tilde{I}$,  
\begin{equation}\label{eq:HxkOriginal}
    H_x ^{\rm T} (x_k, u_k, \lambda_{k+1} ^*) + \phi_{x} ^{\rm T} \nu ^* + \lambda_{k+1} ^* - \lambda_k ^* = 0,
\end{equation}
and (\ref{eq:KKT:Hu}) for $i \in \left\{0,..., N-1 \right\}$, in which (\ref{eq:f})--(\ref{eq:x0}) are trivially satisfied because $x_0, ..., x_N$ and $u_0, ..., u_{N-1}$ satisfy the FONC of the transformed OCP.
It should be noted that because $x_0, ..., x_N$ and $u_0, ..., u_{N-1}$ satisfy (\ref{eq:f}) and $\phi_x (\cdot)$ only depends on the generalized coordinate, 
\begin{align}\label{eq:proof:phix:equivalence}
    \phi_x (x_k) = \phi_x (& x_{k-2} + f(x_{k-2}, u_{k-2}) \notag \\ 
                           & + g(x_{k-2} + f(x_{k-2}, u_{k-2})))
\end{align}
holds. 
Therefore, we describe both the left- and right-hand sides of (\ref{eq:proof:phix:equivalence}) as $\phi_x$ in this proof.
Let
\begin{equation}\label{eq:proof:lmd:star}
    \lambda_i ^* = \lambda_i , \;\; i \in \left\{ 0, ..., k-2, k+1, ..., N \right\}.
\end{equation}
Then, (\ref{eq:KKT:phix}), (\ref{eq:KKT:Hx}) for $i \in \left\{ 0, ..., k-3, k+1, ..., N \right\}$, and (\ref{eq:KKT:Hu}) for $i \in \left\{ 0, ..., k-2, k+1, ..., N \right\}$ of the original OCP are reduced to those of the transformed OCP and are, therefore, satisfied.
Furthermore, let
\begin{equation}\label{eq:proof:nu:star}
    \nu ^* = \nu, \; 
    \lambda_{k} ^* = \lambda_{k} + \phi_{x} ^{\rm T} \nu ^*, \; 
    \lambda_{k-1} ^* = \lambda_{k-1} + (I + g_{x} ^{\rm T} ) \phi_{x} ^{\rm T} \nu ^*.
\end{equation}
Then, (\ref{eq:KKT:Hx}) and (\ref{eq:KKT:Hu}) for $i = k-2, k-1$ and (\ref{eq:HxkOriginal}) of the original OCP are also reduced to (\ref{eq:KKT:Hx:k-2}), (\ref{eq:KKT:Hu:k-2}), (\ref{eq:KKT:Hx}) and (\ref{eq:KKT:Hu}) for $i = k-1$, and (\ref{eq:KKT:Hx}) for $i = k$ of the transformed OCP, respectively, noting that $\phi_x f_{x} (x_{k-1}, u_{k-1}) = \phi_x g_{x}$ and $\phi_x f_u (x_{k-1}, u_{k-1}) = O$, and are, therefore, satisfied, which completes the proof.
\end{proof}

From the proof of Theorem 4.1, we can obtain the Lagrange multipliers at a stationary point of the original OCP corresponding to those of the transformed OCP.
The following theorem concerns the sufficiency of the optimality.

\begin{theorem}
Suppose that $x_0, ..., x_N$, $u_0, ..., u_{N-1}$, $\lambda_0, ..., \lambda_N$, and $\nu$ satisfy the SOSC of the transformed OCP.
Then, there exist the Lagrange multipliers $\lambda_0 ^*, ..., \lambda_N ^*$, and $\nu ^*$ such that $x_0, ..., x_N$, $u_0, ..., u_{N-1}$, $\lambda_0 ^*, ..., \lambda_N ^*$, and $\nu ^*$ satisfy the SOSC of the original OCP. 
\end{theorem}

\begin{proof}
Because $x_0, ..., x_N$, $u_0, ..., u_{N-1}$, $\lambda_0, ..., \lambda_N$, and $\nu$ also satisfy the FONC of the transformed OCP, we have $\lambda_0 ^*, ..., \lambda_N ^*$ and $\nu ^*$ defined by (\ref{eq:proof:lmd:star}) and (\ref{eq:proof:nu:star}) such that $x_0, ..., x_N$, $u_0, ..., u_{N-1}$, $\lambda_0 ^*, ..., \lambda_N ^*$, and $\nu ^*$ satisfy the FONC of the original OCP from Theorem 4.1.
From the assumption of the SOSC of the transformed OCP, we have
\begin{equation}\label{eq:proof:Lagrangian}
    \delta x_N ^{\rm T} Q_{xx, N} \delta x_N  
    + \sum_{i=0}^{N-1} 
    \begin{bmatrix}
        \delta x_i ^{\rm T} \\
        \delta u_i ^{\rm T}
    \end{bmatrix} ^{\rm T}
    \begin{bmatrix}
        Q_{xx, i} & Q_{xu, i} \\
        Q_{ux, i} ^{\rm T} & Q_{uu, i} 
    \end{bmatrix}
    \begin{bmatrix}
        \delta x_i \\
        \delta u_i
    \end{bmatrix} > 0
\end{equation}
for arbitrary $\delta x_i$ and $\delta u_i$ satisfying 
$\delta x_0 = 0$, $(I + f_{x, i}) \delta x_i + f_{u, i} \delta u_i - \delta x_{i+1} = 0$ for $i \in \left\{ 0, ..., N-1 \right\}$ and $\phi_x (I + g_x) (I + f_{x, k-2} ) \delta x_{k-2} + \phi_x (I + g_x) f_{u, k-2} \delta u_{k-2} = 0$, where we describe $f_{x, i} := f_x (x_i, u_i)$ and $f_{u, i} := f_u (x_i, u_i)$ for $i \in \left\{ 0, ..., N-1 \right\}$. 
We introduce the Hessians of the original OCP as $Q_{xx, N} ^* := \varphi_{xx} (x_N) = Q_{xx, N}$
and
$Q_{xx, i} ^* \allowbreak := H_{xx} (x_i, u_i, \lambda_{i+1} ^*)$,
$Q_{xu, i} ^* \allowbreak := H_{xu} (x_i, u_i, \lambda_{i+1} ^*)$,
and
$Q_{uu, i} ^* \allowbreak := H_{uu} (x_i, u_i, \lambda_{i+1} ^*)$
for $i \in \tilde{I}$. 
Further, we introduce
$Q_{xx, k} ^* \allowbreak := H_{xx} (x_k, u_k, \lambda_{k+1} ^*, \nu^*)$,
$Q_{xu, k} ^* \allowbreak := H_{xu} (x_k, u_k, \lambda_{k+1} ^*, \nu^*)$,
and
$Q_{uu, k} ^* \allowbreak := H_{uu} (x_k, u_k, \lambda_{k+1} ^*, \nu^*)$.
We can then complete the proof if 
\begin{align}\label{eq:proof:LagrangianOrigin}
    & {\delta x_N ^*} ^{\rm T} Q_{xx, N} ^* \delta x_N ^* \notag \\ 
    & + \sum_{i=0}^{N-1} 
    \begin{bmatrix}
        {\delta x_i ^*} ^{\rm T} \\
        {\delta u_i ^*} ^{\rm T}
    \end{bmatrix} ^{\rm T}
    \begin{bmatrix}
        Q_{xx, i} ^* & Q_{xu, i} ^* \\
        {Q_{ux, i} ^*} ^{\rm T} & Q_{uu, i} ^*
    \end{bmatrix}
    \begin{bmatrix}
        \delta x_i ^* \\
        \delta u_i ^*
    \end{bmatrix} > 0
\end{align}
holds for arbitrary $\delta x_i ^*$ and $\delta u_i ^*$ satisfying $\delta x_0 ^* = 0$, $(I + f_{x, i}) \delta x_i ^* + f_{u, i} \delta u_i ^* - \delta x_{i+1} ^* = 0$ for $i \in \left\{ 0, ..., N-1 \right\}$
and $\phi_x \delta x_k ^* = 0$.
First, we can see that the subspace of the feasible variations $\delta x_i ^*$ and $\delta u_i ^*$ is identical to that of $\delta x_i$ and $\delta u_i$ because (\ref{eq:proof:phix:equivalence}) holds.
Then, we consider $\delta x_{i} ^*$ and $\delta u_{i} ^*$ as being identical to $\delta x_{i}$ and $\delta u_{i}$.
Next, by substituting (\ref{eq:proof:lmd:star}) and (\ref{eq:proof:nu:star}) into the Hessians of the original OCP, we obtain
$Q_{xx, i} ^* = Q_{xx, i}$,
$Q_{xu, i} ^* = Q_{xu, i}$,
and
$Q_{uu, i} ^* = Q_{uu, i}$
for $i \in \left\{0, ..., k-3, k+1, ..., N-1 \right\}$,
$Q_{xx, k} ^* = Q_{xx, k} + \nu \cdot \phi_{xx}$, $Q_{xu, k} ^* = Q_{xu, k}$, $Q_{uu, k} ^* = Q_{uu, k}$,
$Q_{xx, k-1} ^* = Q_{xx, k-1} + (\phi_x ^{\rm T} \nu) \cdot f_{xx, k-1}$,
$Q_{xu, k-1} ^* = Q_{xu, k-1} + (\phi_x ^{\rm T} \nu) \cdot f_{xu, k-1}$,
$Q_{uu, k-1} ^* = Q_{uu, k-1} + (\phi_x ^{\rm T} \nu) \cdot f_{uu, k-1}$,
\begin{align*}
    & Q_{xx, k-2} ^* = Q_{xx, k-2} \notag \\ 
    & - (I + f_{x, k-2} ^{\rm T}) ((\phi_x ^{\rm T} \nu) \cdot g_{xx}) (I + f_{x, k-2}) \notag \\ 
    & - (I + f_{x, k-2} ^{\rm T}) (I + g_x ^{\rm T}) (\nu \cdot \phi_{xx}) (I + g_x) (I + f_{x, k-2}) ,
\end{align*}
\begin{align*}
    & Q_{xu, k-2} ^* = Q_{xu, k-2} - (I + f_{x, k-2} ^{\rm T}) ((\phi_x ^{\rm T} \nu) \cdot g_{xx}) f_{u, k-2} \notag \\ 
    & - (I + f_{x, k-2} ^{\rm T}) (I + g_x ^{\rm T}) (\nu \cdot \phi_{xx}) (I + g_x) f_{u, k-2}, 
\end{align*}
and
\begin{equation*}
    Q_{uu, k-2} ^* = Q_{uu, k-2} - f_{u, k-2} ^{\rm T} (I + g_x ^{\rm T}) (\nu \cdot \phi_{xx}) (I + g_x) f_{u, k-2},
\end{equation*}
where $f_{xx, i} := f_{xx} (x_i, u_i)$, $f_{xu, i} := f_{xu} (x_i, u_i)$, $f_{uu, i} := f_{uu} (x_i, u_i)$, and the notation ``$\cdot$" denotes vector--tensor multiplication.
Since the FONC of the original OCP holds and 
$\nu \cdot \phi_{xx} = \begin{bmatrix}
    \nu \cdot \phi_{qq} & O \\ 
    O & O 
\end{bmatrix}$,
we have $(\nu \cdot \phi_{xx}) (I + f_{x, k-1}) = (\nu \cdot \phi_{xx}) (I + g_{x})$ and $(\nu \cdot \phi_{xx}) f_{u, k-1} = O$, which yields $\delta x_k ^{\rm T} (\nu \cdot \phi_{xx}) \delta x_k = \delta x_{k-1} ^{\rm T} (I+ g_{x} ^{\rm T}) (\nu \cdot \phi_{xx}) (I+ g_{x}) \delta x_{k-1}$.
In addition, from the structures of (\ref{eq:stateEquationForm}), (\ref{eq:positionLevelConstraints}), and (\ref{eq:g}), we have
$(\phi_x ^{\rm T} \nu) \cdot f_{xx, k-1} = (\phi_x ^{\rm T} \nu) \cdot g_{xx}$,
$(\phi_x ^{\rm T} \nu) \cdot f_{xu, k-1} = O$,
and
$(\phi_x ^{\rm T} \nu) \cdot f_{uu, k-1} = O$.
By substituting these relations, the left-hand side of (\ref{eq:proof:LagrangianOrigin}) is reduced to the left-hand side of (\ref{eq:proof:Lagrangian}), which completes the proof.
\end{proof}

We summarize the property of the proposed transformation in the next proposition:

\begin{proposition}
Suppose that $x_0, ..., x_N$, $u_0, ..., u_{N-1}$, $\lambda_0, ..., \lambda_N$, and $\nu$ satisfy the SOSC of the transformed OCP. 
Then, the solution $x_0, ..., x_N$, $u_0, ..., u_{N-1}$ is a strict local minimum of the original OCP. 
\end{proposition}

\begin{proof}
Because the solution $x_0, ..., x_N$, $u_0, ..., u_{N-1}$ with the Lagrange multipliers satisfies the SOSC of the original OCP, as indicated by Theorem 4.2, the solution $x_0, ..., x_N$, $u_0, ..., u_{N-1}$ is a strict local minimum of the original OCP.
\end{proof}

\section{Numerical Experiments on Whole-Body Quadrupedal Gaits Optimization}\label{section:numericalExperiments}

\subsection{Experimental Settings}
To demonstrate the effectiveness of the proposed method over existing methods, we conducted numerical experiments on the whole-body optimal control of a quadrupedal robot ANYmal for various gaits.
The equation of motion of the full 3D model of the quadrupedal robot is of the form of (\ref{eq:stateEquationForm}).
Moreover, a pure-state constraint whose Jacobian is of the form (\ref{eq:positionLevelConstraints}) is imposed just before each impact between the leg and the ground, which is termed the switching constraint \cite{bib:penalty, bib:DDP:jumpRobot}.
We compare the following three Riccati recursion algorithms based on the direct-multiple shooting method and Gauss-Newton Hessian approximation with various constraint handling methods:
\begin{itemize}
    \item The proposed method
    \item The Riccati recursion with pure-state constraints \cite{bib:constrainedRiccati} 
    \item The AL method \cite{bib:nocedal}
\end{itemize}
We implemented these three algorithms in C++ and used Pinocchio \cite{bib:pinocchio}, an efficient C++ library used for rigid-body dynamics and its analytical derivatives, to compute the dynamics and its derivatives of the quadrupedal robot.
We used OpenMP \cite{bib:OpenMP} for parallel computing (e.g., line 1 of Algorithm 1) and four threads through the following experiments. 
To consider the practical situation, we also imposed inequality constraints on the joint angle limits, joint angular velocity limits, and joint torque limits of each joint.  
We used the primal-dual interior point method \cite{bib:ipopt} with fixed barrier parameters for the inequality constraints.
None of the three methods used line search; they only used the fraction-to-boundary rule \cite{bib:ipopt} for step-size selection.
We fixed the instants of the impact between the robot and the ground in the following experiments and did not treat them as optimization variables as in \cite{bib:penalty, bib:DDP:jumpRobot}, to focus on the evaluation of the constraint-handling methods.
All experiments were conducted on a laptop with a hexa-core CPU Intel Core i9-8950HK @2.90 GHz.

In the following two experiments, we considered that the OCP converges when the $l_2$-norm of the residuals in the Karush—Kuhn--Tucker (KKT) conditions, which we refer to as the KKT error, becomes smaller than a prespecified threshold.
The KKT conditions are composed of the FONC and primal and dual residuals in the inequality constraints.
For example, the KKT conditions of the proposed method are composed of (\ref{eq:f}), (\ref{eq:x0}), (\ref{eq:phi:relaxed}), (\ref{eq:KKT:phix})--(\ref{eq:KKT:Hu:k-2}), and the residuals in the inequality constraints.
The KKT conditions of the Riccati recursion with pure-state constraints \cite{bib:constrainedRiccati} and the AL method are only slightly different depending on the method used to treat pure-state equality constraints.

\subsection{Trotting Gait for Different Numbers of Steps}
First, we evaluated the performances of the three methods for different total dimensions of pure-state equality constraints by considering the trotting gaits of ANYmal with different numbers of trotting steps. 
A six-dimensional (three-dimensional for each impact leg) pure-state equality constraint (switching constraint) was imposed on the OCPs for each trotting step.
We chose the number of trotting steps from 2, 4, 6, 8, and 10 and measured the CPU time per Newton iteration and the total CPU time until convergence 
(we chose $1.0 \times 10^{-10}$ as the convergence tolerance of the KKT errors).
The settings used for the OCPs (horizon length $T$, number of grids $N$, and total dimension of equality constraints (\ref{eq:phi})) are listed in Table I.
We carefully tuned the parameters of the AL method \cite{bib:nocedal}. For example, we chose the initial penalty parameter as $p = 5$ and the penalty update value as $\beta = 8$; that is, the AL method updates the penalty parameter as $p \leftarrow \beta p$ when the KKT error excluding the constraint violation (\ref{eq:phi}) is smaller than a tolerance that is also tuned carefully \cite{bib:nocedal}.

Figure \ref{fig:trotting} depicts the CPU time per Newton iteration (left figure) and the total CPU time until convergence (right figure) of each method.
As depicted in the left figure of Fig. \ref{fig:trotting}, the CPU time per Newton iteration in the proposed method was almost the same as that in the AL method, whereas the Riccati recursion with pure-state constraints \cite{bib:constrainedRiccati} took more computational time when compared with the other two methods.
The right figure of Fig. \ref{fig:trotting} also indicates that the proposed method achieved the fastest convergence.
The proposed method took exactly the same number of iterations (approximately 20) until convergence as the Riccati recursion with pure-state constraints \cite{bib:constrainedRiccati} in all the cases. Therefore, the proposed method was faster than it in terms of the total CPU time as in the case of the per Newton iteration. 
The AL method was significantly slower than the other two methods with respect to the total CPU time because it required approximately 80 iterations in all the cases, although we carefully tuned the AL algorithm.

\begin{table}[tb]
\centering
\caption{
Settings of OCPs for each number of trotting steps
}
\begin{tabular}{|c||r|r|r|r|r|}
\hline
No. of trotting steps & 2    & 4    & 6    & 8    & 10   \\ \hline \hline
Horizon length $T$ & 1.55 & 2.55 & 3.55 & 4.55 & 5.55 \\ \hline
No. of grids $N$ & 35  & 59 & 83 & 107  & 131 \\ \hline
Total dim. of (\ref{eq:phi}) & 12   & 24   & 36   & 48   & 60   \\ \hline
\end{tabular}
\end{table}

\begin{figure}[tb]
    \centering
    \includegraphics[scale=0.635]{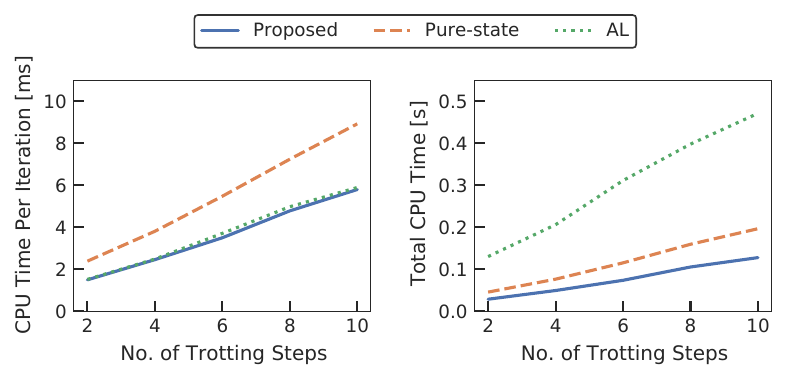}
    \caption{
    (Left) CPU time per Newton iteration and (right) total CPU time until convergence for different numbers of trotting steps in the proposed method (Proposed), Riccati recursion with pure-state constraints \cite{bib:constrainedRiccati} (pure state), and the AL method (AL).
    }
    \label{fig:trotting}
\end{figure}

\subsection{Trotting, Jumping, and Running Gait Problems}
Next, we investigated the performances of the proposed method in three different quadrupedal gaits: trotting, jumping, and running gaits, of which the jumping and running gaits are particularly highly nonlinear and complicated problems.
Each jumping step imposes a 12-dimensional pure-state equality constraint, and each running step imposes a 6-dimensional one.
We summarize the settings of each problem (horizon length $T$, number of grids $N$, number of steps, total dimension of equality constraints (\ref{eq:phi}), tolerance of convergence, and initial penalty parameter of the AL method $p_{\rm init}$) in Table II.
As done in the preceding example, we carefully tuned the parameters of the AL method, that is, $p_{\rm init}$ and the update rule of the penalty and the Lagrange multiplier, for each problem.
We measured the KKT errors with respect to the number of iterations, and the total number of iterations and CPU time until convergence.

Figure \ref{fig:convergence} depicts the $\log_{10}$-scaled KKT error of each method for the three gait problems with respect to the number of iterations.
We can see that the convergence behavior of the proposed method was almost the same as that of the Riccati recursion with pure-state constraints \cite{bib:constrainedRiccati}.
In contrast, the AL method resulted in significantly slow convergence because it needs to update the penalty parameter and Lagrange multiplier to reduce the constraint violation, which we can see in the peaks in the KKT error of the AL method in Fig. \ref{fig:convergence}.
Figure \ref{fig:CPUtime} indicates the number of iterations and total CPU time until convergence of each method.
We can see that the total number of iterations of the proposed method was almost the same as that of the Riccati recursion with pure-state constraints \cite{bib:constrainedRiccati}, whereas the AL method required a significantly large number of iterations. 
In addition, as each iteration of the proposed method was faster than that of the Riccati recursion with pure-state constraints \cite{bib:constrainedRiccati}, as in the previous experiment, the proposed method achieved the fastest convergence.

\begin{table}[tb]
\centering
\caption{
Settings of OCPs for trotting, jumping, and running gaits.
}
\begin{tabular}{|c||r|r|r|}
\hline
Gait type          & Trotting & Jumping & Running \\ \hline \hline
Horizon length $T$ & 6.05     & 5       & 7    \\ \hline
No. of grids $N$   & 143      & 107     & 346     \\ \hline
No. of steps       & 11       & 3       & 26      \\ \hline
Total dim. of (\ref{eq:phi})  & 66      & 36      & 156     \\ \hline
KKT tolerance      & $1.0\times 10 ^{-10}$ & $1.0\times 10 ^{-10}$ & $1.0\times 10 ^{-8}$       \\ \hline
$p_{\rm init}$     & 5        & 1000    & 5       \\ \hline
\end{tabular}
\end{table}

\begin{figure}[tb]
    \centering
    \includegraphics[scale=0.63]{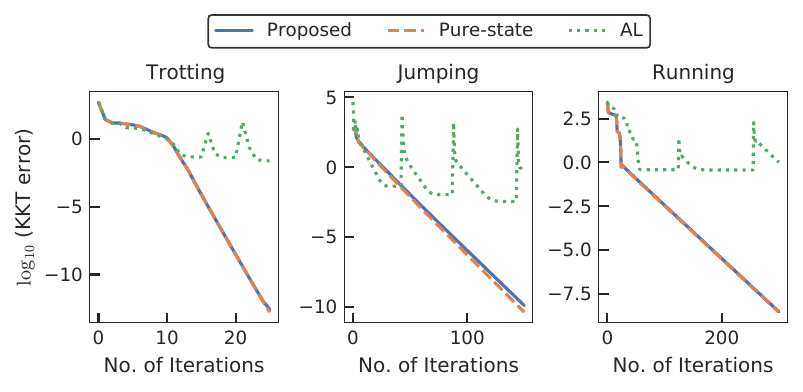}
    \caption{
    $\log_{10}$-scaled KKT errors of the proposed method (Proposed), the Riccati recursion with pure-state constraints \cite{bib:constrainedRiccati} (Pure-state), and the AL method (AL) for the three gait problems with respect to the number of iterations. 
    It should be noted that the KKT errors include the violations of the switching constraints.
    The graphs of the AL method have peaks when the penalty parameter and Lagrange multipliers are updated.
    }
    \label{fig:convergence}
\end{figure}

\begin{figure}[tb]
    \centering
    \includegraphics[scale=0.63]{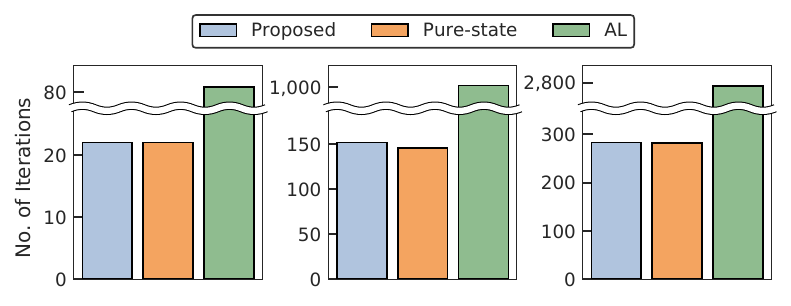}
    \includegraphics[scale=0.63]{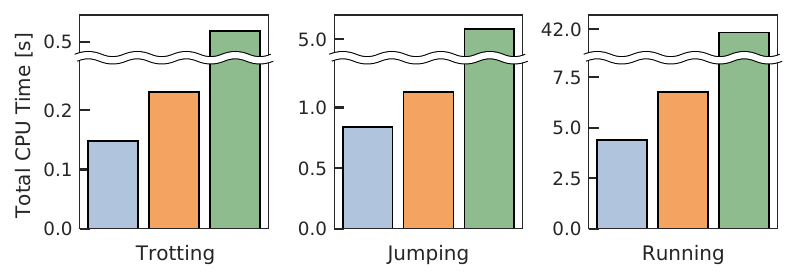}
    \caption{
    Number of iterations and total CPU time until convergence of the proposed method (Proposed), the Riccati recursion with pure-state constraints \cite{bib:constrainedRiccati} (Pure-state), and the AL method (AL) for the three gaits problems.
    }
    \label{fig:CPUtime}
\end{figure}

\section{Conclusions}\label{section:conclu}

We proposed a novel approach to efficiently treat pure-state equality constraints in OCPs with a Riccati recursion algorithm.
The proposed method transforms a pure-state equality constraint into a mixed state-control constraint such that the constraint is expressed by variables at a certain previous time stage.
We derived a Riccati recursion algorithm to solve the transformed OCP with linear time complexity in the grid number of the horizon, in contrast to the previous approach \cite{bib:constrainedRiccati}, which scaled cubically with respect to the total dimension of the pure-state equality constraints.
Because the proposed method is essentially a Newton's method for an optimization problem with equality constraints, the proposed method achieves superlinear or quadratic convergence, which distinguishes our approach from the penalty function method and the AL method in terms of the convergence property.
We showed that if the solution satisfies the FONC and/or SOSC of the transformed OCP, then the solution also satisfies the FONC and/or SOSC of the original OCP.
Therefore, if we find a solution that satisfies the SOSC of the transformed OCP, it is a local minimum of the original OCP.
We performed numerical experiments on the whole-body optimal control of quadrupedal gaits that involve pure-state equality constraints owing to contact switches and demonstrated the effectiveness of the proposed method over the approach of \cite{bib:constrainedRiccati} and the AL method. 

Our future work will include applying the proposed method with switching time optimization problems \cite{bib:penalty, bib:DDP:jumpRobot}.
We further extend the proposed method to constraints whose relative degree is larger than 2.


\section*{Acknowledgment}
This work was partly supported by JST SPRING, Grant Number JPMJSP2110.





\bibliographystyle{IEEEtran}
\bibliography{IEEEabrv, ieee}

\end{document}